\documentclass[10pt,oneside]{amsart}

\usepackage[
paper=a4paper,
headsep=15pt,text={135mm,216mm},centering
]{geometry}

\usepackage[bookmarks]{hyperref}
\usepackage{amssymb,amsxtra}
\usepackage{graphicx,xcolor,mathrsfs}
\usepackage{float,array,calc,booktabs,stackrel,url}
\usepackage{pb-diagram} 
\usepackage{enumitem}\setlist[enumerate,1]{font=\upshape}\setlist[enumerate,2]{font=\upshape}
\usepackage{mathtools}
\usepackage{tikz}
\usetikzlibrary{
  cd,
  calc,
  positioning,
  arrows,
  decorations.pathreplacing,
  decorations.markings,
}
\usepackage[mode=buildnew]{standalone}

\definecolor{todo-background-color}{gray}{0.95}
\usepackage[
  textwidth=1.1in,
  backgroundcolor=todo-background-color,
  bordercolor=black,
  linecolor=black,
  textsize=footnotesize
]{todonotes}

\iftrue
    
    \makeatletter
    \def\@settitle{%
      \vspace*{-10pt}
      \begin{flushleft}%
        \LARGE\bfseries
        \strut\@title\strut
      \end{flushleft}%
    }
    \def\@setauthors{%
      \begingroup
      \def\thanks{\protect\thanks@warning}%
      \trivlist
      \raggedright
      \large \@topsep27\p@\relax
      \advance\@topsep by -\baselineskip
    \item\relax
      \author@andify\authors
      \def\\{\protect\linebreak}%
      \authors
      \ifx\@empty\contribs
      \else
      ,\penalty-3 \space \@setcontribs
      \@closetoccontribs
      \fi
      \normalfont
      \endtrivlist
      \endgroup
    }
    \def\@setaddresses{\par
      \nobreak \begingroup
      \small\raggedright
      \def\author##1{\nobreak\addvspace\smallskipamount}%
      \def\\{\unskip, \ignorespaces}%
      \interlinepenalty\@M
      \def\address##1##2{\begingroup
        \par\addvspace\bigskipamount\noindent
        \@ifnotempty{##1}{(\ignorespaces##1\unskip) }%
        {\ignorespaces##2}\par\endgroup}%
      \def\curraddr##1##2{\begingroup
        \@ifnotempty{##2}{\nobreak\noindent\curraddrname
          \@ifnotempty{##1}{, \ignorespaces##1\unskip}\/:\space
          ##2\par}\endgroup}%
      \def\email##1##2{\begingroup
        \@ifnotempty{##2}{\nobreak\noindent E-mail address%
          \@ifnotempty{##1}{, \ignorespaces##1\unskip}\/:\space
          \ttfamily##2\par}\endgroup}%
      \def\urladdr##1##2{\begingroup
        \def~{\char`\~}%
        \@ifnotempty{##2}{\nobreak\noindent\urladdrname
          \@ifnotempty{##1}{, \ignorespaces##1\unskip}\/:\space
          \ttfamily##2\par}\endgroup}%
      \addresses
      \endgroup
      \global\let\addresses=\@empty
    }
    \def\@setabstracta{%
      \ifvoid\abstractbox
      \else
      \skip@17pt \advance\skip@-\lastskip
      \advance\skip@-\baselineskip \vskip\skip@
      \box\abstractbox
      \prevdepth\z@ 
      \vskip-28pt
      \fi
    }
    \renewenvironment{abstract}{%
      \ifx\maketitle\relax
      \ClassWarning{\@classname}{Abstract should precede
        \protect\maketitle\space in AMS document classes; reported}%
      \fi
      \global\setbox\abstractbox=\vtop \bgroup
      \normalfont\small
      \list{}{\labelwidth\z@
        \leftmargin0pc \rightmargin\leftmargin
        \listparindent\normalparindent \itemindent\z@
        \parsep\z@ \@plus\p@
        
      }%
    \item[\hskip\labelsep\bfseries\abstractname.]%
    }{%
      \endlist\egroup
      \ifx\@setabstract\relax \@setabstracta \fi
    }

    \def\ps@headings{\ps@empty
      \def\@evenhead{%
        \setTrue{runhead}%
        \normalfont\scriptsize
        \rlap{\thepage}\hfill
        \def\thanks{\protect\thanks@warning}%
        \leftmark{}{}}%
      \def\@oddhead{%
        \setTrue{runhead}%
        \normalfont\scriptsize
        \def\thanks{\protect\thanks@warning}%
        \rightmark{}{}\hfill \llap{\thepage}}%
      \let\@mkboth\markboth
    }\ps@headings

    \def\section{\@startsection{section}{1}%
      \z@{-1.4\linespacing\@plus-.5\linespacing}{.8\linespacing}%
      {\normalfont\bfseries\Large}}
    \def\subsection{\@startsection{subsection}{2}%
      \z@{-.8\linespacing\@plus-.3\linespacing}{.5\linespacing\@plus.2\linespacing}%
      {\normalfont\bfseries\large}}
    \def\subsubsection{\@startsection{subsubsection}{3}%
      \z@{.7\linespacing\@plus.2\linespacing}{-1.5ex}%
      {\normalfont\itshape}}
    \def\paragraph{\@startsection{paragraph}{4}%
      \z@{.7\linespacing\@plus.2\linespacing}{-1.5ex}%
      {\normalfont\itshape}}
    \def\@secnumfont{\bfseries}

    \renewcommand\contentsnamefont{\bfseries}
    \def\@starttoc#1#2{\begingroup
      \setTrue{#1}%
      \par\removelastskip\vskip\z@skip
      \@startsection{}\@M\z@{\linespacing\@plus\linespacing}%
      {.5\linespacing}{
        \contentsnamefont}{#2}%
      \ifx\contentsname#2%
      \else \addcontentsline{toc}{section}{#2}\fi
      \makeatletter
      \@input{\jobname.#1}%
      \if@filesw
      \@xp\newwrite\csname tf@#1\endcsname
      \immediate\@xp\openout\csname tf@#1\endcsname \jobname.#1\relax
      \fi
      \global\@nobreakfalse \endgroup
      \addvspace{32\p@\@plus14\p@}%
      \let\tableofcontents\rela\x
    }
    \def\contentsname{Contents}
    \def\l@section{\@tocline{2}{.5ex}{0mm}{5pc}{}}
    \def\l@subsection{\@tocline{2}{0pt}{2em}{5pc}{}}
    \makeatother

\fi

\def\to{\mathchoice{\longrightarrow}{\rightarrow}{\rightarrow}{\rightarrow}}
\makeatletter
\newcommand{\shortxra}[2][]{\ext@arrow 0359\rightarrowfill@{#1}{#2}}
\def\longrightarrowfill@{\arrowfill@\relbar\relbar\longrightarrow}
\newcommand{\longxra}[2][]{\ext@arrow 0359\longrightarrowfill@{#1}{#2}}
\renewcommand{\xrightarrow}[2][]{\mathchoice{\longxra[#1]{#2}}%
  {\shortxra[#1]{#2}}{\shortxra[#1]{#2}}{\shortxra[#1]{#2}}}
\makeatother

\makeatletter
\def\addtagsub#1{\let\oldtf=\tagform@\def\tagform@##1{\oldtf{##1}\hbox{$_{#1}$}}}
\makeatother

\makeatletter
\def\Nopagebreak{\@nobreaktrue\nopagebreak}
\makeatother

\makeatletter
\newtheoremstyle{theorem-giventitle}
        {}{}              
        {\itshape}                      
        {}                              
        {\bfseries}                     
        {.}                             
        {\thm@headsep}                             
        {\thmnote{\bfseries#3}}
\newtheoremstyle{theorem-givenlabel}
        {}{}              
        {\itshape}                      
        {}                              
        {\bfseries}                     
        {.}                             
        {\thm@headsep}                             
        {\thmname{#1}~\thmnumber{#3}\setcurrentlabel{#3}}
\newtheoremstyle{definition-giventitle}
        {}{}              
        {}                      
        {}                              
        {\bfseries}                     
        {.}                             
        {\thm@headsep}                             
        {\thmnote{\bfseries#3}}
\def\setcurrentlabel#1{\gdef\@currentlabel{#1}}
\makeatother

\newtheorem{theorem}{Theorem}[section]

\newtheorem{proposition}[theorem]{Proposition}
\newtheorem{corollary}[theorem]{Corollary}
\newtheorem{lemma}[theorem]{Lemma}

\newtheorem{question}[theorem]{Question}

\theoremstyle{definition}

\newtheorem{remark}[theorem]{Remark}

\newtheorem*{case2'}{Case 2$'$}

\theoremstyle{theorem-giventitle}
\newtheorem{theorem-named}{}
\theoremstyle{theorem-givenlabel}
\newtheorem{theorem-labeled}{Theorem}

\theoremstyle{definition-giventitle}
\newtheorem{definition-named}{}
\newtheorem{conjecture-named}{}
\newtheorem{case-named}{}

\numberwithin{equation}{section}

\def\Z{\mathbb{Z}}
\def\R{\mathbb{R}}
\def\Q{\mathbb{Q}}

\DeclareMathOperator\Ker{Ker}

\def\id{\mathrm{Id}}

\def\csum{\mathbin{\#}}

\def\Bl{\mathop{B\ell}}
\def\rhot{\rho^{(2)}}

\def\cC{\mathcal{C}}

\def\Fix{\mbox{Fix}}
\def\tor{\mbox{torsion}}
\makeatletter

\begin{document}

\title[knot reversal]{Knot reversal and rational concordance}

\author{Taehee Kim}
\address{Department of Mathematics\\
  Konkuk University\\
  Seoul 05029\\
  Republic of Korea
}
\email{tkim@konkuk.ac.kr}


\def\subjclassname{\textup{2010} Mathematics Subject Classification}
\expandafter\let\csname subjclassname@1991\endcsname=\subjclassname
\expandafter\let\csname subjclassname@2000\endcsname=\subjclassname
\subjclass{%
  57M25, 
  57M27, 
  57N13, 
  57N70, 
}
\keywords{rational concordance, knot reverse, slice knot}

\begin{abstract}
We give an infinite family of knots that are not rationally concordant to their reverses. More precisely, if $\tau$ denotes the involution of the rational knot concordance group $\cC_\Q$ induced by string reversal and $\Fix(\tau)$ denotes the subgroup of knots fixed under $\tau$ in $\cC_\Q$, then $\cC_\Q/\Fix(\tau)$ contains an infinite rank subgroup. As a corollary, we show that there exists a knot $K$ such that for every pair of coprime integers $p$ and $q$, the $(p,q)$-cable of $K$ is not concordant to the reverse of the $(p,q)$-cable of $K$. 
\end{abstract}
\maketitle

\section{Introduction}\label{section:introduction}

For an oriented knot $K$ in $S^3$, the {\it reverse} of $K$ is the knot obtained by reversing the string orientation of $K$. In some earlier literature, it was called the inverse of $K$; in this paper the {\it inverse} of $K$ is the mirror image of $K$ with reversed string orientation, which is denoted by $-K$.  

Let $\tau(K)$ denote the reverse of $K$. It is not easy to distinguish a knot $K$ from $\tau(K)$ since manifolds obtained from $K$ and $\tau(K)$ are diffeomorphic: the $p/q$-surgeries on $K$ and $\tau(K)$ are diffeomorphic, and the $n$-fold branched cyclic covers of $S^3$ over $K$ and $\tau(K)$ are diffeomorphic for each $n$. The infinite cyclic covers of the complements of $K$ and $\tau(K)$ in $S^3$ are also diffeomorphic. Nonetheless, Trotter \cite{Trotter:1963-1} proved that there exists a knot that is not equivalent to its reverse. 

It is a more subtle problem to distinguish a knot from its reverse up to concordance. Using extensions of Casson-Gordon invariants \cite{Gilmer:1983-1} and two-fold branched covers of knots, Livingston \cite{Livingston:1983-1} showed that there are knots that are not concordant to their reverses. (Later a gap was found in the poof  of a statement that is given in \cite{Gilmer:1983-1} and used in \cite{Livingston:1983-1}. It is known that the proof of \cite{Livingston:1983-1} can be completed using three-fold branched covers of knots, instead.) Kirk and Livingston \cite{Kirk-Livingston:1999-3} showed that the knot $8_{17}$ is not concordant to its reverse using twisted Alexander polynomials, which are also extensions of Casson-Gordon invariants. Recently, Livingston and the author \cite{Kim-Livingston:2019-1} showed that there are {\it topologically slice} knots that are not concordant to their reverses.

In this paper, we give the first example of knots that are not {\it rationally} concordant to their reverses. Recall that two knots $K_0$ and $K_1$ are {\it concordant} if there is a properly embedded smooth annulus in $S^3\times [0,1]$ with boundary $K_1\times\{1\}\sqcup -K_0\times \{0\}$. Two knots $K_1$ and $K_0$ in $S^3$ are {\it rationally concordant} if there exist a rational homology cobordism $W^4$ with boundary $S^3\times \{1\}\sqcup -S^3\times \{0\}$ and a properly embedded smooth annulus $\Sigma$ in $W$ with $\partial \Sigma = K_1\times \{1\} \sqcup -K_0\times \{0\}$. Said differently, two knots are rationally concordant if they are concordant in a rational homology cobordism. A knot $K$ in $S^3$ is {\it rationally slice} if it is rationally concordant to the unknot. Equivalently, A knot is rationally slice if and only it bounds a smoothly embedded disk in a rational homology 4-ball with boundary $S^3$. Note that a knot $K$ is rationally concordant to $\tau(K)$ if and only if $K\csum -\tau(K)$ is rationally slice.

Rational concordance is an equivalence relation, and modulo rational concordance of knots in $S^3$, we obtain the {\it rational knot concordance group} under connected sum, which we denote by $\cC_\Q$. The involution $\tau$ induces an involution on $\cC_\Q$, which we denote by $\tau$ as well. Let $\Fix(\tau)$ denote the subgroup of knots fixed under $\tau$ in $\cC_\Q$. The following is our main theorem.

\begin{theorem}\label{theorem:main}
The group $\cC_\Q/\Fix(\tau)$ contains an infinite rank subgroup. In particular, there exist knots that are not rationally concordant to their reverses.
\end{theorem}

Cochran, Davis, and Ray \cite[Corollary~5.3]{Cochran-Davis-Ray:2014-1} showed that for coprime integers $p$ and $q$ and knots $K_0$ and $K_1$, the $(p,q)$-cable of $K_0$ is concordant to the $(p,q)$-cable of $K_1$ in a $\Z[1/p]$-homology $S^3\times [0,1]$ if and only if $K_0$ is concordant to $K_1$ in a $\Z[1/p]$-homology $S^3\times [0,1]$. Since the reverse of the $(p,q)$-cable of a knot $K$ is equivalent to the $(p,q)$-cable of the reverse of $K$, combined with the above observation Theorem~\ref{theorem:main} immediately gives the following corollary.

\begin{corollary}\label{corollary:main}
	There exists a knot $K$ such that for every pair of coprime integers $p$ and $q$, the $(p,q)$-cable of $K$ is not concordant to the reverse of the $(p,q)$-cable of $K$. 
\end{corollary}

Kearton \cite{Kearton:1989-1} showed that for every knot $K$, the knot $K\csum -\tau(K)$ is a (negative) mutant of the slice knot $K\csum -K$. Therefore the knots in \cite{Livingston:1983-1, Kirk-Livingston:1999-3} that are not concordant to their reverses give us examples of knots that are not concordant to their (negative) mutants. Since the knots $K$ in Theorem~\ref{theorem:main} are not {\it rationally} concordant to their reverses, we obtain the following corollary that extends the aforementioned result to rational concordance.

\begin{corollary}\label{corollary:mutation}
There exists a knot $K$ that is not rationally concordant to a (negative) mutant of $K$. 
\end{corollary}
\noindent Put differently, mutation may change the rational concordance class of a knot.

We note that Kirk and Livingston \cite{Kirk-Livingston:1999-3, Kirk-Livingston:1999-1} showed that there are knots that are not concordant to their positive mutants. The author does not know whether or not the knots in Corollary~\ref{corollary:mutation} are rationally concordant to their positive mutants, and the following natural question arises:

\begin{question}\label{question:positive-mutant}
	Does there exist a knot that is not rationally concordant to its positive mutant?
\end{question}

To prove Theorem~\ref{theorem:main} we use the von Neumann $\rhot$-invariant (see Theorem~\ref{theorem:obstruction}) that was introduced as a knot concordance invariant by Cochran, Orr, and Teichner \cite{Cochran-Orr-Teichner:2003-1} and later developed further as a rational concordance invariant by Cha \cite{Cha:2007-1} (see Section~\ref{section:obstructions}). In this paper, we do not use {\it higher-order} $\rhot$-invariants but only use $\rhot$-invariants associated with {\it metabelian} representations (see Remark~\ref{remark:metabelian}), which have many analogies with Casson-Gordon invariants. For example, Casson-Gordon invariants use $\Q/\Z$-valued linking forms on finite cyclic covers of $S^3$ branched over a knot, and metabelian $\rhot$-invariants use the (rational) Blanchfield form on the infinite cyclic cover of the 0-surgery on a knot in $S^3$. 

Compared with Casson-Gordon invariants, an advantage of using metabelian $\rhot$-invariants, which we take in this is paper, is that we are required to deal with only one cover (the infinite cyclic cover), whereas Casson-Gordon invariants might necessitate considering infinitely many finite covers; as {\it rational} sliceness obstructions, metabelian $\rhot$-invariants still use the infinite cyclic cover only, and to resolve the technical difficulties in the setting of rational concordance we only need to track down changes of module structures on the Alexander modules and the corresponding Blanchfield forms. (See Section~\ref{section:obstructions}.)

We note that the von Neumann $\rhot$-invariants are topological concordance invariants, and therefore all our results in this paper are also available in the topological rational knot concordance group. This observation gives us the following question.

\begin{question}\label{question:topologically slice}
	Does there exist a {\it topologically slice} knot that is not rationally concordant to its reverse?
\end{question}

This paper is organized as follows. In Section~\ref{section:obstructions} we review the metabelian $\rhot$-invariants and rational sliceness obstructions (Theorem~\ref{theorem:obstruction}). In Section~\ref{section:single-example} we give a single knot that  is not rationally concordant to its reverse, giving the key ideas of the proof of Theorem~\ref{theorem:main}. Finally, we prove Theorem~\ref{theorem:main} in Section~\ref{section:proof}.

In this paper all manifolds are oriented and homology groups are understood with integer coefficients unless mentioned otherwise. 

\subsection*{Acknowledgments} 
The author thanks Charles Livingston for helpful comments and conversations. 

\section{Rational sliceness obstructions}\label{section:obstructions}
In this section, we review rational sliceness obstructions given in \cite{Cochran-Orr-Teichner:2003-1,Cha:2007-1}; all the details of materials given in this section can be found in \cite[Section~5]{Cha:2007-1}. 

For a knot $K$, let $M_K$ denote the 0-surgery on $K$ in $S^3$. If $K$ is rationally slice, then there exists a 4-manifold $W$ with $H_*(W;\Q)\cong H_*(S^1;\Q)$ and $\partial W=M_K$. For, if $K$ bounds a smooth disk $D$ in a rational homology 4-ball $V$, the exterior $V\setminus \nu D$, where $\nu D$ is the open tubular neighborhood of $D$ in $V$, gives the desired $W$. In particular, $H_1(M_K;\Q)\cong H_1(W;\Q)\cong \Q$. Note that the group $H_1(M_K)$ has no torsion, and the group $H_1(M_K)=H_1(M_K)/\tor$ is isomorphic to $\Z=\langle s\rangle$ where the generator $s$ is represented by a meridian of $K$. The group $H_1(W)/\tor$ is also isomorphic to $\Z$ and we denote a generator by $t$. That is, $H_1(W)/\tor\cong \Z=\langle t\rangle$. 

The inclusion-induced homomorphism $i_*\colon H_1(M_K)\to H_1(W)/\tor$ sends $s$ to $t^c$ for some nonzero integer $c$, and we may assume $c$ is positive by changing orientaions if necessary. Then, we call the positive integer $c$ the {\it complexity} of $V$ (and $W$). We also say that $K$ is {\it slice in $V$ of complexity $c$}.

Suppose that $K$ is slice in a rational homology 4-ball $V$ of complexity $c$. Then for the 4-manifold $W$ as defined above and the inclusion $i\colon M_K\hookrightarrow W$, we obtain a representation $\epsilon_c\colon \pi_1M_K\to \langle s\rangle \xrightarrow{i_*}\langle t\rangle$ whose image is $\langle t^c\rangle$. Via this representation, we obtain a $\Q[t^{\pm 1}]$-module $H_1(M_K;\Q[t^{\pm 1}])$. Since $s$ is mapped to $t^c$ under $i_*$, we may consider $\Q[t^{\pm 1}]$ as a $\Q[s^{\pm 1}]$-module where $s$ acts as multiplication by $t^c$ and we have the isomorphism
\[
H_1(M_K;\Q[t^{\pm 1}]) \cong H_1(M_K;\Q[s^{\pm 1}])\otimes_{\Q[s^{\pm 1}]} \Q[t^{\pm 1}] 
\]
where $H_1(M_K;\Q[s^{\pm 1}])$ is the rational Alexander module of $K$.

On $ H_1(M_K;\Q[t^{\pm 1}])$, we have the (nonsingular) Blanchfield form 
\[
\Bl\nolimits_c\colon H_1(M_K;\Q[t^{\pm 1}])\otimes H_1(M_K;\Q[t^{\pm 1}]) \to \Q(t)/\Q[t^{\pm 1}]. 
\]
For the rational Blanchfield form 
\[
\Bl\colon H_1(M_K;\Q[s^{\pm 1}])\otimes H_1(M_K;\Q[s^{\pm 1}]) \to \Q(s)/\Q[s^{\pm 1}]
\]
and $x, y\in H_1(M_K;\Q[s^{\pm 1}])$ and $f(t),g(t)\in \Q[t^{\pm 1}]$, we have 
\[
\Bl\nolimits_c(x\otimes f(t), y\otimes g(t)) = f(t)\cdot h(\Bl(x,y))\cdot g(t^{-1})
\] 
where the map $h\colon \Q(s)/\Q[s^{\pm 1}] \to \Q(t)/\Q[t^{\pm 1}]$ is induced from $s\mapsto t^c$.

Notice that $\Z=\langle t\rangle$ acts on $\Q(t)/\Q[t^{\pm 1}]$ by multiplication and from this action we obtain a semidirect product $\left(\Q(t)/\Q[t^{\pm 1}]\right)\rtimes \Z$.

Let $\epsilon\colon \pi_1M_K\to H_1(M_K)=\Z=\langle s\rangle$ be the abelianization. For each $y\in \pi_1M_K$, let $\epsilon'(y)=n$ if $\epsilon(y)=s^n$. Let $\mu$ be a meridian of $K$ such that $\epsilon(\mu)=s$. Then for each $y\in \pi_1M_K$, we obtain a homology class $y\mu^{-\epsilon'(y)} \in H_1(M_K;\Q[s^{\pm 1}])$. Now for each $x\in H_1(M_K;\Q[t^{\pm 1}])$, we obtain a representation 
\[
\phi_x\colon \pi_1M_K\to \Q(t)/\Q[t^{\pm 1}]\rtimes \langle t\rangle
\] 
defined by $\phi_x(y)= \left(\Bl_c(x,(y\mu^{-\epsilon'(y)})\otimes 1), \epsilon_c(y)\right)$ and the von-Neumann $\rhot$-invariant $\rhot(K,\phi_x)\in \R$ associated with $\phi_x$. (For a definition of the von-Neumann $\rhot$-invariant, which we will not need in this paper, refer to \cite[Section~5]{Cochran-Orr-Teichner:2003-1} and \cite[Sections~2 and 3]{Cha:2007-1}.)

\begin{remark}\label{remark:metabelian}
The representation $\phi_x$ is called {\it metabelian} since it factors through the second derived quotient of $\pi_1M_K$: let $\pi=\pi_1M_K$. Let $\pi^{(1)}=[\pi,\pi]$, the commutator subgroup of $\pi$, and let $\pi^{(2)}=[\pi^{(1)},\pi^{(1)}]$. Then we have a short exact sequence
\[
0\to \pi^{(1)}/ \pi^{(2)} \to \pi/ \pi^{(2)}\to \pi/ \pi^{(1)} \to 0.
\]
Since $\pi/ \pi^{(1)}\cong H_1(M_K)\cong \Z=\langle s \rangle$, the above short exact sequence splits, and it follows that the second derived quotient $\pi/ \pi^{(2)}$ is isomorphic to $\pi^{(1)}/ \pi^{(2)} \rtimes \langle s\rangle$.  Since $ \pi^{(1)}/ \pi^{(2)}\cong H_1(M_K;\Z[s^{\pm 1}])$, we obtain an isomorphism $\pi/ \pi^{(2)} \cong H_1(M_K;\Z[s^{\pm 1}])\rtimes \langle s\rangle$. Then, for a map $i_c\colon \langle s\rangle \to \langle t\rangle$ defined by $i_c(s^n)=t^{cn}$, the above representation $\phi_x$ is the composition of maps
\begin{align*}
\pi \to \pi/ \pi^{(2)} 	&\xrightarrow{\cong} H_1(M_K;\Z[s^{\pm 1}])\rtimes \langle s\rangle \\
				&\xrightarrow{\phantom{\cong}} H_1(M_K;\Q[s^{\pm 1}])\rtimes \langle s\rangle \\
				&\xrightarrow{(\otimes 1, i_c)} H_1(M_K;\Q[t^{\pm 1}])\rtimes \langle t\rangle \\
				&\xrightarrow{(\Bl_c(x,-), \id)} \Q(t)/\Q[t^{\pm 1}] \rtimes \langle t\rangle.
\end{align*}
The representation $\phi_x$ and the von-Neumann $\rhot$-invariant associated with $\phi_x$ were first introduced in \cite{Cochran-Orr-Teichner:2003-1} for the case $c=1$ and in \cite{Cha:2007-1} for the case $c>1$. In \cite{Cochran-Orr-Teichner:2003-1, Cha:2007-1} they also introduce non-metabelian representations, which do not factor through $\pi/\pi^{(2)}$, and the corresponding {\it higher-order} $\rhot$-invariants, but we do not use them in this paper. 
\end{remark}

For a $\Q[t^{\pm 1}]$-submodule $P$ of $H_1(M_K;\Q[t^{\pm 1}])$, let 
\[
P^\perp=\{x\in H_1(M_K;\Q[t^{\pm 1}])\,\mid\, \Bl\nolimits_c(x,y)=0 \mbox{ for all }y\in P\}.
\]
We say that $P$ is {\it self-annihilating} if $P=P^\perp$. The following is our rational sliceness obstruction which will be used in the proof of Theorem~\ref{theorem:main}.

\begin{theorem}[{\cite[Theorem~4.6]{Cochran-Orr-Teichner:2003-1} for $c=1$ and \cite[Theorem~5.13]{Cha:2007-1} for $c>1$}]\label{theorem:obstruction}
 Suppose that a knot $K$ in $S^3$ is slice in a rational homology 4-ball $V$ of complexity $c$. Let $t$ denote the generator of the group $H_1(V\setminus \nu D)/\tor$ where $D$ is a slice disk for $K$ in $V$. Then, there exists a self-annihilating submodule $P$ of $H_1(M_K;\Q[t^{\pm 1}])$ such that $\rhot(K,\phi_x)=0$ for all $x\in P$.
\end{theorem}

\section{A single example}\label{section:single-example}
In this section, to show the key ideas of the proof of Theorem~\ref{theorem:main} more clearly, we construct a single example $K$ and show that it is not rationally concordant to its reverse. 

For a knot $K$ in $S^3$ and the abelianization $\epsilon\colon \pi_1M_K\to H_1(M_K)=\Z=\langle s\rangle$, for brevity let $\rhot_0(K)=\rhot(K,\epsilon)$. By \cite[Lemma~5.3]{Cochran-Orr-Teichner:2004-1}, we have $\rhot_0(K)=\int_{S^1} \sigma_K(z)\, dz$ where $\sigma_K$ denotes the Levine-Tristram signature function for $K$ and $S^1$ is normalized to length one (that is, $\int_{S^1}dz=1$). Let $J_\alpha$ and $J_\beta$ be two knots such that $\rhot_0(J_\alpha)\ne 0, \rhot_0(J_\beta)\ne 0$, and $\rhot_0(J_\alpha)\ne \rhot_0(J_\beta)$. As a specific choice, we let  $J_\alpha$ and $J_\beta$ be the left-handed trefoil and the right-handed trefoil, respectively. (Then $\rhot_0(J_\alpha)=2/3$ and $\rhot_0(J_\beta)=-2/3$.) 

\begin{figure}[htb!]
	\includegraphics{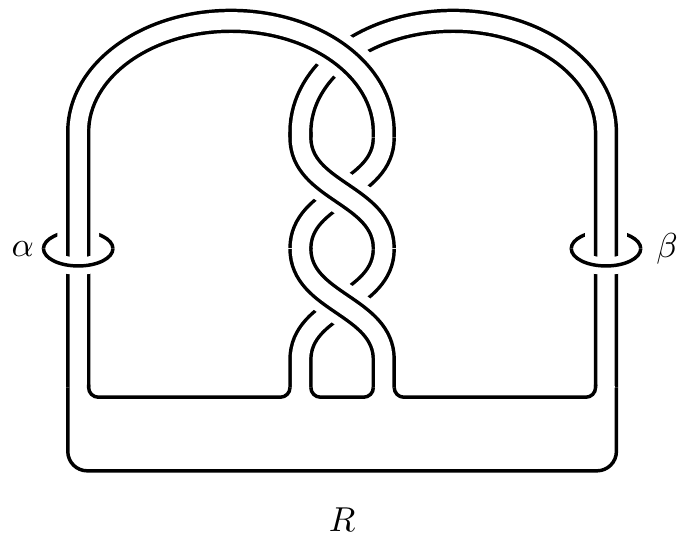}
	\caption{The knot $R$}\label{figure:knot-R}
\end{figure}

Let $R$ be the knot $9_{46}$ as depicted in Figure~\ref{figure:knot-R}. In Figure~\ref{figure:knot-R}, the curve $\alpha$ is a curve dual to the left band of the obvious Seifert surface of $R$ and $\beta$ is a curve dual to the right band of the Seifert surface. For knots $J_\alpha$ and $J_\beta$, let us denote by $R(\alpha, \beta; J_\alpha, J_\beta)$ the knot obtained via the satellite construction which ties the knot $J_\alpha$ (respectively, $J_\beta$) through the band dual to the curve $\alpha$ (respectively, $\beta$).

\begin{theorem}\label{theorem:single-example} Let $K=R(\alpha, \beta; J_\alpha, J_\beta)$. The knot $K$ is not rationally concordant to $\tau(K)$. That is, $K$ is nontrivial in $\cC_\Q/\Fix(\tau)$.
\end{theorem}

\begin{proof}
Suppose that $K$ is rationally concordant to $\tau(K)$ and let $L=K\# -\tau(K)$. Then, $L$ bounds a slice disk $D$ in a rational homology 4-ball $V$ of complexity $c$ for some positive integer $c$. Let $W=V\setminus \nu D$. Then $\partial W=M_L$ and the inclusion-induced map $H_1(M_L)=\Z=\langle s\rangle \to H_1(W)/\tor=\Z=\langle t\rangle$ sends $s$ to $t^c$. Now we have 
\[
H_1(M_K;\Q[s^{\pm 1}]) \cong H_1(M_R;\Q[s^{\pm 1}])\cong \frac{\Q[s^{\pm 1}]}{\langle 2s-1 \rangle} \oplus \frac{\Q[s^{\pm 1}]}{\langle s-2 \rangle}
\] 
where we may assume that $\alpha$ (respectively, $\beta$) generates $\Q[s^{\pm 1}]/\langle 2s-1\rangle$ (respectively, $\Q[s^{\pm 1}]/\langle s-2\rangle$). 

When we take the inverse of $R$, we reverse both the orientation of the ambient space $S^3$ and the string orientation of $R$, and therefore meridians of $R$ are inverted twice. It follows that an oriented meridian of $R$ is also an oriented meridian of $-R$, and $M_R$ and $M_{-R}$ have diffeomorphic infinite cyclic covers with the same deck transformations. On the other hand, meridians of $R$ are inverted once for taking the reverse of $R$, and therefore $M_R$ and $M_{\tau(R)}$ have diffeomorphic infinite cyclic covers with the inverted deck transformations. Consequently, $M_R$ and $M_{-\tau(R)}$ have diffeomorphic infinite cyclic covers with the inverted deck transformations. Noticing that $-\tau(K) = (-\tau(R))(\alpha', \beta'; -\tau(J_\alpha), -\tau(J_\beta))$ where $\alpha'$ and $\beta'$ denote the images of $\alpha$ and $\beta$ for $-\tau(R)$, respectively, we can deduce that
\[
H_1(M_{-\tau(K)};\Q[s^{\pm 1}]) \cong H_1(M_{-\tau(R)};\Q[s^{\pm 1}])\cong \frac{\Q[s^{\pm 1}]}{\langle s-2 \rangle} \oplus \frac{\Q[s^{\pm 1}]}{\langle 2s-1 \rangle}
\]
where $\alpha'$ (respectively, $\beta'$) generates $\Q[s^{\pm 1}]/\langle s-2\rangle$ (respectively, $\Q[s^{\pm 1}]/\langle 2s-1\rangle$). 

Finally, we have 
\begin{align*}
H_1(M_L;\Q[s^{\pm 1}])
&\cong H_1(M_K;\Q[s^{\pm 1}]) \oplus H_1(M_{-\tau(K)};\Q[s^{\pm 1}])\\
&\cong \frac{\Q[s^{\pm 1}]}{\langle 2s-1 \rangle} \oplus  \frac{\Q[s^{\pm 1}]}{\langle s-2 \rangle} \oplus \frac{\Q[s^{\pm 1}]}{\langle s-2 \rangle} \oplus \frac{\Q[s^{\pm 1}]}{\langle 2s-1 \rangle}\\
&= \langle \alpha\rangle \oplus \langle \beta\rangle \oplus \langle \alpha'\rangle \oplus \langle \beta'\rangle.
\end{align*}

Since $H_1(M_L;\Q[t^{\pm 1}])\cong H_1(M_L;\Q[s^{\pm 1}])\otimes_{\Q[s^{\pm 1}]} \Q[t^{\pm 1}]$ and $s$ is mapped to $t^c$, we have 
\begin{align*}
	H_1(M_L;\Q[t^{\pm 1}]) 	&\cong \frac{\Q[t^{\pm 1}]}{\langle 2t^c-1 \rangle} \oplus  \frac{\Q[t^{\pm 1}]}{\langle t^c-2 \rangle} \oplus \frac{\Q[t^{\pm 1}]}{\langle t^c-2 \rangle} \oplus \frac{\Q[t^{\pm 1}]}{\langle 2t^c-1 \rangle}\\
	&= \langle \alpha\otimes 1\rangle \oplus \langle \beta\otimes 1\rangle \oplus \langle \alpha'\otimes 1\rangle \oplus \langle \beta'\otimes 1\rangle.
\end{align*}

By Theorem~\ref{theorem:obstruction}, there exists a self-annihilating submodule $P$ of $H_1(M_L;\Q[t^{\pm 1}])$ such that $\rhot(L,\phi_x)=0$ for each $x\in P$ and the representation $\phi_x\colon \pi_1M_K\to \Q(t)/\Q[t^{\pm 1}]\rtimes \langle t\rangle$ associated with $x$. Since $P=P^\perp$ and the Blanchfield form $\Bl_c$ is nonsingular, the submodule $P$ is nontrivial. 

Let $x=(a,b,a',b')$ be a nontrivial element in $P$. Suppose that $a\ne 0$ in $\Q[t^{\pm 1}]/\langle 2t^c-1\rangle = \langle \alpha\otimes 1\rangle$. We may assume this by changing orientations if necessary. Since $P$ is a $\Q[t^{\pm 1}]$-submodule and $x\in P$, the element $(t^c-2)\cdot x$ also lies in $P$. Also notice that $2t^c-1$ and $t^c-2$ are mutually prime in $\Q[t^{\pm 1}]$, and therefore $(t^c-2)\cdot a$ is nontrivial in $\Q[t^{\pm 1}]/\langle 2t^c-1\rangle = \langle \alpha\otimes 1\rangle$. Therefore, by letting $t^c-2$ act on $x$ if necessary, we may assume $x=(a,0,0,b')\in P$ and $a\ne 0$.

From a standard cobordism argument (for example, see \cite[pages 1750103-7 and 1750103-8]{Kim:2017-1}), it follows that the representation $\phi_x$ induces representations 
\[
\phi_{(a,0)}\colon \pi_1M_K\to \Q(t)/\Q[t^{\pm 1}]\rtimes \langle t\rangle \mbox{ and } \phi_{(0,b')}\colon \pi_1M_{-\tau(K)}\to \Q(t)/\Q[t^{\pm 1}]\rtimes \langle t\rangle
\] 
and
\[
\rhot(L,\phi_x)= \rhot(K,\phi_{(a,0)})+\rhot(-\tau(K),\phi_{(0,b')}).
\]
Since $x=(a,0,0,b')\in P$, it follows that $\rhot(L,\phi_x)=0$, and therefore 
\[
\rhot(K,\phi_{(a,0)})+\rhot(-\tau(K),\phi_{(0,b')}) = 0.
\]
We will show that $\rhot(K,\phi_{(a,0)})=-2/3$ and $\rhot(K,\phi_{(0,b')})=0$ or $-2/3$, which will contradict the above equation and finish the proof. 

We show that $\rhot(K,\phi_{(a,0)})=-2/3$. Since $\alpha\otimes 1$ generates $\Q[t^{\pm 1}]/\langle 2t^c-1\rangle$ and $a\ne 0$, we may write $a=(\alpha\otimes 1)\cdot f(t) = \alpha\otimes f(t)$ for some $f(t)\in \Q[t^{\pm 1}]$ with $f(t)\notin \langle 2t^c-1\rangle$.  Then, as shown in Section~\ref{section:obstructions}, it follows that 
\[
\phi_{(a,0)}(y)= \left(\Bl\nolimits_c(\alpha\otimes f(t),(y\mu^{-\epsilon'(y)})\otimes 1), \epsilon_c(y)\right).
\]
Recall that the map
\[
\phi_\alpha\colon \pi_1M_K\to \Q(s)/\Q[s^{\pm 1}]\rtimes \langle s\rangle 
\] 
is defined by 
\[
\phi_\alpha(y)= \left(\Bl(\alpha,y\mu^{-\epsilon'(y)}), \epsilon(y)\right).
\]

We will show that $\rhot(K,\phi_{(a,0)})=\rhot(K,\phi_\alpha)$. This fact essentially follows from Lemma~\ref{lemma:subgroup property} below, and we will give more details. (Also refer to \cite[Theorem~5.16]{Cha:2007-1}).)

\begin{lemma}[{\cite[(2.3) (subgroup property) on p.108]{Cochran-Orr-Teichner:2004-1}}]\label{lemma:subgroup property} Let $M$ be a 3-manifold and let $\Gamma$ and $\Gamma'$ be groups. If $\phi\colon \pi_1M\to \Gamma'$ is a homomorphism and $h\colon \Gamma'\to \Gamma$ is an injective homomorphism, then $\rhot(M,\phi) = \rhot(M,h\circ \phi)$.
\end{lemma}

First, we observe that the map $\phi_{(a,0)}\colon \pi_1M_K\to \Q(t)/\Q[t^{\pm 1}]\rtimes \langle t\rangle$ is decomposed as $\phi_{(a,0)} = j\circ f_*\circ \bar{h}\circ {\psi_\alpha}$ where the maps $\psi_\alpha$, $\bar{h}$, $f_*$, and $j$ are defined as follows:

\begin{enumerate}
	\item Let $S=\{(2s-1)^n\mid n\in \Z\}$, a multiplicative closed subset of $\Q[s^{\pm 1}]$, and let $S^{-1}\Q[s^{\pm 1}]$ be the localization of $\Q[s^{\pm 1}]$ by $S$. Since $\alpha$ is $(2s-1)$-torsion in $H_1(M_K;\Q[s^{\pm1 }])$, the image of $\Bl(\alpha, -)$ lies in $S^{-1}\Q[s^{\pm 1}]/\Q[s^{\pm 1}]$. It follows that for the canonical inclusion
	\[
	i\colon S^{-1}\Q[s^{\pm 1}]/\Q[s^{\pm 1}]\rtimes \langle s\rangle\to \Q(s)/\Q[s^{\pm 1}]\rtimes \langle s\rangle 
	\]
	 there is a map
	\[
	\pi_1M_K\to S^{-1}\Q[s^{\pm 1}]/\Q[s^{\pm 1}]\rtimes \langle s\rangle
	\]
	whose composition with the map $i$ is $\phi_\alpha$. We denote the map by $\psi_\alpha$.  That is, $\phi_\alpha=i\circ \psi_\alpha$.
	\item Let $S' =\{(2t^c-1)^n\mid n\in \Z\}$, a multiplicative closed subset of $\Q[t^{\pm 1}]$. We define a map  
	\[
	\bar{h}\colon S^{-1}\Q[s^{\pm 1}]/\Q[s^{\pm 1}]\rtimes \langle s\rangle \to S'^{-1}\Q[t^{\pm 1}]/\Q[t^{\pm 1}]\rtimes \langle t\rangle
	\]
	to be the map induced from $s\mapsto t^c$. That is, $\bar{h}(p(s)+\Q[s^{\pm 1}], s^n)=(p(t^c)+\Q[t^{\pm 1}], t^{cn})$ for $p(s)\in S^{-1}\Q[s^{\pm 1}]$ and $n\in \Z$.
	\item We define a map
	\[
	f_*\colon S'^{-1}\Q[t^{\pm 1}]/\Q[t^{\pm 1}]\rtimes \langle t\rangle \to S'^{-1}\Q[t^{\pm 1}]/\Q[t^{\pm 1}]\rtimes \langle t\rangle
	\] 
	by $f_*(p(t)+\Q[t^{\pm 1}], t^n) = (f(t)p(t)+\Q[t^{\pm 1}], t^n)$ for $p(t)\in S'^{-1}\Q[t^{\pm 1}]$ and $n\in \Z$.
	\item We define a map
	\[
	j\colon S'^{-1}\Q[t^{\pm 1}]/\Q[t^{\pm 1}]\rtimes \langle t\rangle \to \Q(t)/\Q[t^{\pm 1}]\rtimes \langle t\rangle
	\] to be the canonical inclusion.	
\end{enumerate}
One can easily see that $\phi_{(a,0)} = j\circ f_*\circ \bar{h}\circ \psi_\alpha$ and that the maps $\bar{h}$ and $j$ are injective. Since $f(t)\notin \langle 2t^c-1\rangle$ and $2t^c-1$ is irreducible in $\Q[t^{\pm 1}]$, the map $f_*$ is also injective. It follows that the composition $j\circ f_* \circ \bar{h}$ is injective. Now by Lemma~\ref{lemma:subgroup property}, it follows that $\rhot(K,\phi_{(a,0)}) = \rhot(K,\psi_\alpha)$. Since the map $i$ in (1) above is also injective and $\phi_\alpha=i\circ\psi_\alpha$, again by Lemma~\ref{lemma:subgroup property} we have $\rhot(K,\psi_\alpha)=\rhot(K,\phi_\alpha)$. Therefore, $\rhot(K,\phi_{(a,0)}) = \rhot(K,\phi_\alpha)$. 

We compute $\rhot(K,\phi_\alpha)$. Let $R'=R(\alpha;J_\alpha)$. Then $K=R'(\beta;J_\beta)$. Since $K$ and $R'$ have isomorphic Alexander modules and Blanchfield forms, the map $\phi_\alpha$ induces a map 
\[
\phi_\alpha'\colon \pi_1M_{R'}\to \Q(s)/\Q[s^{\pm 1}]\rtimes \langle s\rangle
\] 
defined in the same way as $\phi_\alpha$ on $\pi_1M_K$.

From another standard cobordism argument (for example, see \cite[Lemma~3.3]{Kim:2017-1}, \cite[Lemma~5.22]{Cha:2007-1}, and \cite[Lemma~2.3]{Cochran-Harvey-Leidy:2009-01}), it follows that the map $\phi_\alpha$ induces a map $\phi_{J_\beta}\colon \pi_1M_{J_\beta}\to \Z$ and
\[
\rhot(K, \phi_\alpha) = \rhot(R', \phi_\alpha') + \rhot(J_\beta, \phi_{J_\beta})
\]
where
\[
\rhot(J_\beta,\phi_{J_\beta})=
\begin{cases}
	0 & \mbox{if } \phi_{\alpha}(\beta) \mbox{ is trivial},\\[1ex]
	\rhot_0(J_\beta)=\int_{S^1} \sigma_{J_\beta}(z) dz=-\frac{2}{3} & \mbox{otherwise}.
	
\end{cases}
\]
Note that $R'$ is slice in the 4-ball $D^4$ since there is a surgery curve for a slice disk, say $D'$, on the ``$\alpha$-band" of the obvious Seifert surface for $R'$. Since 
\[
\Ker\{H_1(M_{R'};\Q[s^{\pm 1}])\to H_1(D^4\setminus \nu D';\Q[s^{\pm 1}])\}= \langle \alpha\rangle,
\] 
the map $\phi_\alpha'$ extends to $\pi_1D^4\setminus\nu D'$ if $\phi_\alpha'(\alpha)$ is trivial. Since $R'$ has a Seifert matrix $\begin{psmallmatrix}0 & 1\\ 2 & 0\end{psmallmatrix}$, for the Blanchfield form $\Bl'$ of $R'$, one can easily show that $\Bl'(\alpha, \alpha)=0$, and hence $\phi_\alpha'(\alpha)$ is trivial. Therefore, $\phi_\alpha'$ extends to $\pi_1D^4\setminus \nu D'$. By \cite[Theorem~4.2]{Cochran-Orr-Teichner:2003-1}, it follows that $\rhot(R', \phi_\alpha')=0$.

The knot $K$ has the same Seifert forms as $R'$ (and $R$), and one can easily show that $\Bl(\alpha,\beta)\ne 0$. It follows that $\phi_\alpha(\beta)$ is nontrivial and $\rhot(J_\beta,\phi_{J_\beta})=\rhot_0(J_\beta) = -2/3$. Therefore, \[
\rhot(K,\phi_{(a,0)}) = \rhot(K,\phi_\alpha)=0-2/3=-2/3.
\]

One can compute $\rhot(-\tau(K),\phi_{(0,b')})$ using similar arguments as above: if $b'\ne 0$, then we may write $b'=( \beta'\otimes 1)\cdot g(t)=\beta'\otimes g(t)$ for some $g(t)\notin \langle 2t^c-1\rangle$ and we obtain 
\[
\rhot(-\tau(K),\phi_{(0,b')})=\rhot_0(-\tau(J_\alpha))= - \rhot_0(\tau(J_\alpha))=-\rhot_0(J_\alpha)=-2/3.
\] 

On the other hand, if $b'=0$, then by (2.5) in \cite[p.108]{Cochran-Orr-Teichner:2004-1} it follows that  
\[
\rhot(-\tau(K),\phi_{(0,b')})=0.
\] 
Therefore,
\[
\rhot(-\tau(K),\phi_{(0,b')})=
\begin{cases}
	0 & \mbox{if } b'=0,\\[1ex]
	-\rhot_0(J_\alpha)=-\frac23 & \mbox{if } b'\ne 0.	
\end{cases}
\]

Now we have
\begin{align*}
	\rhot(L,\phi_x) &= \rhot(K,\phi_{(a,0)})+\rhot(-\tau(K),\phi_{(0,b')})\\
					&= -2/3\mbox{ or } -4/3,
\end{align*}
which contradicts that $\rhot(L,\phi_x)=0$.
\end{proof}

\section{Proof of Theorem~\ref{theorem:main}}\label{section:proof}
In this section we give a proof of Theorem~\ref{theorem:main}. For positive integers $i=1, 2, 3,\ldots$, let $J_\alpha^i$ and $J_\beta^i$ be knots such that $\{\rhot_0(J_\alpha^i),\, \rhot_0(J_\beta^i)\}_{i\ge 1}$ is a set of real numbers that are linearly independent over $\Z$. Such knots $J_\alpha^i$ and $J_\beta^i$ exist due to \cite[Proposition~2.6]{Cochran-Orr-Teichner:2004-1}. For each $i$ we let $K_i=R(\alpha, \beta; J_\alpha^i,J_\beta^i)$. Theorem~\ref{theorem:main} follows immediately from the following proposition.
\begin{proposition}\label{proposition:infinite-example}
	The knots $K_i$ are linearly independent in $\cC_\Q/\Fix(\tau)$.
\end{proposition}
\begin{proof}
	One can easily see that $K_i$ are linearly independent in $\cC_\Q/\Fix(\tau)$ if and only if $K_i\# -\tau(K_i)$ are linearly independent in $\cC_\Q$. (For example, see \cite[Lemma~4.4]{Kim-Livingston:2019-1}.)
	
	Let $L=\#_in_i\left(K_i\#-\tau(K_i)\right)$ be a finite linear combination of $K_i\#-\tau(K_i)$ with $n_i$ nonzero integers and suppose that $L$ bounds a slice disk $D$ in a rational homology 4-ball $V$ of complexity $c$. Let $W=V\setminus \nu D$. Then, the inclusion $M_L\hookrightarrow W$ induces the homomorphism $H_1(M_L)=\Z=\langle s\rangle \to H_1(W)/\tor=\Z=\langle t\rangle$ that sends $s$ to $t^c$. For $1\le j\le |n_i|$, let $\alpha_{i,j}$ and $\beta_{i,j}$ (respectively, $\alpha_{i,j}'$ and $\beta_{i,j}'$) be copies of $\alpha$ and $\beta$ for the $j$th copy of $K_i$ (respectively, $-\tau(K_i)$). As we observed in Section~\ref{section:single-example}, for a knot $K$, the spaces $M_K$ and $M_{-K}$ have diffeomorphic infinite cyclic covers with the same deck transformations, and $M_K$ and $M_{-\tau(K)}$ have diffeomorphic infinite cyclic covers with the inverted deck transformations. 
Therefore, we may assume that 
\begin{align*}
	H_1(M_L;\Q[s^{\pm 1}])
	&\cong \bigoplus_i|n_i|\left(H_1(M_{K_i};\Q[s^{\pm 1}]) \oplus H_1(M_{-\tau(K_i)};\Q[s^{\pm 1}])\right)\\
	&\cong \bigoplus_i|n_i|\left(\frac{\Q[s^{\pm 1}]}{\langle 2s-1 \rangle} \oplus  \frac{\Q[s^{\pm 1}]}{\langle s-2 \rangle} \oplus \frac{\Q[s^{\pm 1}]}{\langle s-2 \rangle} \oplus \frac{\Q[s^{\pm 1}]}{\langle 2s-1 \rangle}\right)\\
	&= \bigoplus_i\bigoplus_{1\le j\le |n_i|} \left(\langle \alpha_{i,j}\rangle \oplus \langle \beta_{i,j}\rangle \oplus \langle \alpha_{i,j}'\rangle \oplus \langle \beta_{i,j}'\rangle \right).
\end{align*}
Similarly as in Section~\ref{section:single-example}, $H_1(M_L;\Q[t^{\pm 1}])\cong H_1(M_L;\Q[s^{\pm 1}])\otimes_{\Q[s^{\pm 1}]} \Q[t^{\pm 1}]$ and we have 
\begin{align*}
	H_1(M_L;\Q[t^{\pm 1}])
	&\cong \bigoplus_i|n_i|\left(\frac{\Q[t^{\pm 1}]}{\langle 2t^c-1 \rangle} \oplus  \frac{\Q[t^{\pm 1}]}{\langle t^c-2 \rangle} \oplus \frac{\Q[t^{\pm 1}]}{\langle t^c-2 \rangle} \oplus \frac{\Q[t^{\pm 1}]}{\langle 2t^c-1 \rangle}\right)\\
	&= \bigoplus_i\bigoplus_{1\le j\le |n_i|} \left(\langle \alpha_{i,j}\otimes 1\rangle \oplus \langle \beta_{i,j}\otimes 1\rangle \oplus \langle \alpha_{i,j}'\otimes 1\rangle \oplus \langle \beta_{i,j}'\otimes 1\rangle \right).
\end{align*}

By Theorem~\ref{theorem:obstruction}, there exists a self-annihilating $\Q[t^{\pm 1}]$-submodule $P$ of $H_1(M_L;\Q[t^{\pm 1}])$ such that $\rhot(L,\phi_x)=0$ for each $x\in P$ and the metabelian representation $\phi_x\colon \pi_1M_L\to \Q(t)/\Q[t^{\pm 1}]\rtimes \langle t\rangle$ associated with $x$. Since $P$ is self-annihilating and the Blanchfield form $\Bl_c$ on $H_1(M_L;\Q[t^{\pm 1}])$ is nonsingular, the submodule $P$ is nontrivial. 

Let $x$ be a nontrivial element in $P$. We may write 
\[
x=\bigoplus_i\bigoplus_{1\le j\le |n_i|}(a_{i,j}, b_{i,j}, a_{i,j}', b_{i,j}') 
\]
and at least one of $a_{i,j}, b_{i,j}, a_{i,j}'$, and $b_{i,j}'$ is nonzero. By reordering $K_i$ and changing orientations if necessary, we may assume that $n_1>0$ and $a_{1,1}\ne 0$. Since $P$ is a $\Q[t^{\pm 1}]$-submodule, the element $(t^c-2)\cdot x$ lies in $P$. Since $2t^c-1$ and $t^c-2$ are mutually prime in $\Q[t^{\pm 1}]$, the element $(t^c-2)\cdot a_{1,1}$ is nonzero in $\Q[t^{\pm 1}]/\langle 2t^c-1\rangle = \langle \alpha_{1,1}\otimes 1\rangle$. Therefore, by letting $t^c-2$ act on $x$ if necessary, we may assume that $a_{1,1}\ne 0$ and $b_{i,j}=a_{i,j}'=0$ for all $i,j$. 

Using a standard cobordism argument as in Section~\ref{section:single-example} (refer to see \cite[Lemma~3.3]{Kim:2017-1}, \cite[Lemma~5.22]{Cha:2007-1}, and \cite[Lemma~2.3]{Cochran-Harvey-Leidy:2009-01}), we obtain representations 
\[
\phi_{(a_{i,j},0)}\colon \pi_1M_{K_i}\to \Q(t)/\Q[t^{\pm 1}]\rtimes \langle t\rangle 
\mbox{ and }
\phi_{(0,b_{i,j}')}\colon \pi_1M_{-\tau(K_i)}\to \Q(t)/\Q[t^{\pm 1}]\rtimes \langle t\rangle
\] 
for each $i,j$ and the equation
\[
\rhot(L,\phi_x)= \sum_i \sum_{j=1}^{|n_i|}\delta_i\left(\rhot\left(K_i,\phi_{(a_{i,j},0)}\right)+\rhot\left(-\tau(K_i),\phi_{(0,b_{i,j}')}\right)\right),
\]
where $\delta_i=1$ if $n_i>0$ and $\delta_i=-1$ if $n_i<0$. Since $x\in P$, it follows that $\rhot(L,\phi_x)=0$.

Each of $\rhot(K_i,\phi_{(a_{i,j},0)})$ and $\rhot(-\tau(K),\phi_{(0,b_{i,j}')})$ can be computed using the arguments as in Section~\ref{section:single-example} so that we obtain 
\[
\rhot(K_i,\phi_{(a_{i,j},0)})=
\begin{cases}
	0 & \mbox{if } a_{i,j}=0,\\[1ex]
	\rhot_0(J_\beta^i) & \mbox{if } a_{i,j}\ne 0,	
\end{cases}
\]
and similarly, 
\[
\rhot(-\tau(K_i),\phi_{(0,b_{i,j}')})=
\begin{cases}
	0 & \mbox{if } b_{i,j}'=0,\\[1ex]
	-\rhot_0(J_\alpha^i) & \mbox{if } b_{i,j}'\ne 0.	
\end{cases}
\]
Therefore, we have 
\[
0= \sum_i \sum_{j=1}^{|n_i|}\delta_i\left(\delta_{i,j}\rhot_0(J_\beta^i)-\delta_{i,j}'\rhot_0(J_\alpha^i))\right),
\]
where $\delta_i=\pm 1$ and $\delta_{i,j}$ and $\delta_{i,j}'$ are either 0 or 1. But since $n_1>0$ and $a_{1,1}\ne 0$, it follows that $\delta_1=\delta_{1,1}=1$. Therefore, the right hand side of the above equality is a nontrivial linear combination of $\rhot_0(J_\alpha^i)$ and $\rhot_0(J_\beta^i)$, and it contradicts that  $\{\rhot_0(J_\alpha^i),\, \rhot_0(J,\beta^i)\}_{i\ge 1}$ is a set of real numbers that are linearly independent over $\Z$.
\end{proof}

\bibliographystyle{amsalpha}
\def\MR#1{}
\bibliography{research}

\end{document}